\documentclass[12pt,reqno]{amsart}
\usepackage{amsmath, amsfonts, amssymb, amsthm, hyperref,cite}
\usepackage{bm}
\usepackage{mathrsfs}
\usepackage{enumitem}
\allowdisplaybreaks[4]
\def\l{\left}
\def\r{\right}
\def\bg{\bigg}
\def\({\bg(}
\def\){\bg)}
\def\t{\text}
\def\f{\frac}

\def\eq{\equiv}

\def\d{\mathrm{d}}

\def\N{\mathbb N}

\def\<{\langle}
\def\>{\rangle}
\def\1{{\bf 1}}

\theoremstyle{plain}
\newtheorem{theorem}{Theorem}
\newtheorem{conjecture}{Conjecture}

\newtheorem{lemma}{Lemma}
\newtheorem{corollary}{Corollary}
\theoremstyle{definition}
\newtheorem*{Ack}{Acknowledgment}

\theoremstyle{remark}

\numberwithin{equation}{section}

\begin{document}

\title[On some conjectural supercongruences involving the sequence $t_n(x)$]{On some conjectural supercongruences involving \\ the sequence $t_n(x)$}

\author[Hui-Li Han]{Hui-Li Han}
\address{Department of Applied Mathematics, Nanjing Forestry University, Nanjing 210037, People's Republic of China}
\email{mintcrescent@163.com}

\author[Chen Wang]{Chen Wang$^{\ast}$}
\address{Department of Applied Mathematics, Nanjing Forestry University, Nanjing 210037, People's Republic of China}
\email{cwang@smail.nju.edu.cn}

\begin{abstract}
In this paper, we study some supercongruences involving the sequence
$$
t_n(x)=\sum_{k=0}^n\binom{n}{k}\binom{x}{k}\binom{x+k}{k}2^k
$$
and solve some open problems. For any odd prime $p$  and $p$-adic integer $x$, we determine $\sum_{n=0}^{p-1}t_n(x)^2$ and $\sum_{n=0}^{p-1}(n+1)t_n(x)^2$ modulo $p^2$; for example, we establish that
		\begin{align*}
			\sum_{n=0}^{p-1}t_n(x)^2\equiv\begin{cases}
				\left(\dfrac{-1}{p}\right)\pmod{p^2},&\text{if }2x\equiv-1\pmod{p},\\[8pt]
				(-1)^{\langle x\rangle_p}\dfrac{p+2(x-\langle x\rangle_p)}{2x+1}\pmod{p^2},&\text{otherwise,}
			\end{cases}
		\end{align*}
where $\langle x\rangle_p$ denotes the least nonnegative residue of $x$ modulo $p$. This confirms a conjecture of Z.-W. Sun.
\end{abstract}

\thanks{$^{\ast}$Corresponding author}

\maketitle

\section{Introduction}

In 2006, during their study of special values of spectral zeta functions, Kimoto and Wakayama \cite{KW} introduced the Ap\'ery-like numbers:
$$
\tilde{J}_2(n):=\sum_{k=0}^n\binom{n}{k}(-1)^k\binom{-1/2}{k}^2\quad (n\in\N=\{0,1,2,\ldots\})
$$
and conjectured that for any odd prime $p$,
\begin{equation}\label{KWconj}
\sum_{n=0}^{p-1}\tilde{J}_2(n)^2\eq \l(\f{-1}{p}\r)\pmod{p^3},
\end{equation}
where $(\f{\cdot}{p})$ stands for the Legendre symbol. This conjecture was confirmed by Long, Osburn and Swisher \cite{LOS} in 2016.

Motivated by the above work, for any $n\in\N$, Z.-W. Sun \cite{SunZW2017} introduced the following polynomial:
$$
S_n(x,y):=\sum_{k=0}^n\binom{n}{k}\binom{x}{k}\binom{-1-x}{k}y^k.
$$
In particular, denote $S_n(x,-1)$ by $s_n(x)$. Clearly, $\tilde{J}_2(n)=s_n(-1/2)$. Sun established several interesting supercongruences involving $s_n(x)$, such as
\begin{equation}\label{suncong}
\sum_{n=0}^{p-1}s_n(x)^2\eq (-1)^{\<x\>_p}\f{p+2(x-\<x\>_p)}{2x+1}\pmod{p^2},
\end{equation}
where $p>3$ is a prime, $x$ is a $p$-adic integer with $x\not\eq-1/2\pmod{p}$ and $\<x\>_p$ denotes the least nonnegative residue of $x$ modulo $p$. Meanwhile, Sun \cite[Conjecture 6.10]{SunZW2017} conjectured that \eqref{suncong} even holds modulo $p^3$, which was confirmed by the second author and Wang \cite{Wang2025}. Moreover, Sun \cite[Conjecture 6.11]{SunZW2017} also conjectured a modulus $p^4$ extension of \eqref{KWconj} and it was later proved by Liu \cite[Theorem 1.2]{Liu2018}.

In this paper, we mainly focus on congruences concerning the sequence
$$S_n(x,-2)=\sum_{k=0}^n\binom{n}{k}\binom{x}{k}\binom{x+k}{k}2^k$$
which was also denoted by $t_n(x)$ in \cite{SunZW2017}. Our first goal is to show the following supercongruence conjectured by Sun \cite[Conjecture 6.14(i)]{SunZW2017}.

\begin{theorem}\label{mainth1}
	For any odd prime and $p$-adic integer $x$, we have
		\begin{align}\label{mainth1eq}
			\sum_{n=0}^{p-1}t_n(x)^2\equiv\begin{cases}
				\left(\dfrac{-1}{p}\right)\pmod{p^2},&\text{if}\ 2x\equiv-1\pmod{p},\\[8pt]
				(-1)^{\langle x\rangle_p}\dfrac{p+2(x-\langle x\rangle_p)}{2x+1}\pmod{p^2},&\text{otherwise.}
			\end{cases}
		\end{align}
\end{theorem}

Sun \cite[Conjecture 6.14(iii)]{SunZW2017} also proposed the following conjecture.

\begin{conjecture}\label{sunconj}
Let $p$ be an odd prime. Then
   \begin{align}
		\sum_{n=0}^{p-1}(8n+5)t_n\left(-\dfrac{1}{2}\right)^2 &\equiv 2p \pmod{p^2},\label{suncon1}\\
		\sum_{n=0}^{p-1}(32n+21)t_n\left(-\dfrac{1}{4}\right)^2 &\equiv 8p \pmod{p^2}.\label{suncon2}
	\end{align}
If $p>3$, then
\begin{align}
		\sum_{n=0}^{p-1}(18n+7)t_n\left(-\dfrac{1}{3}\right)^2 &\equiv 0 \pmod{p^2},\label{suncon3}\\
		\sum_{n=0}^{p-1}(72n+49)t_n\left(-\dfrac{1}{6}\right)^2 &\equiv 18p \pmod{p^2}.\label{suncon4}
\end{align}
\end{conjecture}

In view of Theorem \ref{mainth1}, it remains to evaluate $\sum_{n=0}^{p-1}nt_n(x)^2$ modulo $p^2$ for $x=-1/2$, $-1/3,-1/4,-1/6$. To archive this goal, we establish the following result.

\begin{theorem}\label{mainth2}
Let $p$ be an odd prime. Then, for any $p$-adic integer $x\not\equiv -1/2\pmod{p}$, we have
		\begin{equation}\label{mainth2_1}
	\sum_{n=0}^{p-1}(n+1)t_n(x)^2\equiv\f{p}{4}-\f{(-1)^{\<x\>_p}(2x^2+2x-1)(p+2(x-\<x\>_p))}{8x+4}\pmod{p^2}.
\end{equation}
Moreover, for any $p$-adic integer $x\equiv -1/2\pmod{p}$, we have
\begin{equation}\label{mainth2_2}
	\sum_{n=0}^{p-1}(n+1)t_n(x)^2\equiv\frac{p}{4}+\frac{3}{8}	\left(\frac{-1}{p}\right)\pmod{p^2}.
\end{equation}
\end{theorem}

\begin{corollary}\label{sunconm}
Conjecture \ref{sunconj} is true.
\end{corollary}

For more congruence properties of the polynomial sequence $S_n(x,y)$, one may consult \cite{Guo,Liu2017,SunZH2022,WangZhong}. We shall prove Theorem \ref{mainth1} in the next section. The proofs of Theorem \ref{mainth2} and Corollary \ref{sunconm} will be given in Section 3. Interestingly, in order to prove the two theorems, we find closed formulas for the following double sums:
$$
\sum_{k=0}^n\sum_{l=0}^{n}\binom{n}{k}\binom{n+k}{k}\binom{n}{l}\binom{n+l}{l}\dfrac{(-2)^{k+l}}{(k+l+1)\binom{k+l}{k}}
$$
and
$$
\sum_{k=0}^n\sum_{l=0}^{n}\binom{n}{k}\binom{n+k}{k}\binom{n}{l}\binom{n+l}{l}\dfrac{(-2)^{k+l}}{\binom{k+l+2}{k+1}},
$$
which are also important in their own rights.

\section{Proof of Theorem \ref{mainth1}}

Let $p$ be an odd prime. In the following contents, for any $p$-adic integer $x$, we always use $m$ to denote $\<x\>_p$ and write $x=m+pt$, where $t$ is a $p$-adic integer. In order to prove Theorem \ref{mainth1}, we need the following lemmas. The first one concerns the $p$-adic expansion of $\binom{x}{k}\binom{x+k}{k}$ and plays a key role in the subsequent proofs.
\begin{lemma}[Wang and Wang {\cite[Lemma 2.1]{Wang2025}}]\label{xkx}
For any odd prime $p$ and $p$-adic integer $x$ with $m\leq(p-1)/2$, we have
\begin{align*}
\binom{x}{k}\binom{x+k}{k}\eq\begin{cases}\binom{m}{k}\binom{m+k}{k}(1+ptH_{m+k}-ptH_{m-k})\pmod{p^2},\quad&\t{if}\ 0\leq k\leq m,\vspace{8pt}\\ 0\pmod{p^2},\quad&\t{if}\ p-m\leq k\leq p-1,\end{cases}
\end{align*}
where $H_n:=\sum_{k=1}^n1/k$ stands for harmonic number.

Moreover, if $m<(p-1)/2$ and $m+1\leq k\leq p-1-m$, then we have
\begin{align*}
\binom{x}{k}\binom{x+k}{k}\eq\f{(-1)^{m+k+1}pt\binom{m+k}{k}}{(k-m)\binom{k}{m}}\pmod{p^2}.
\end{align*}
\end{lemma}

\begin{lemma}\label{identity3}
	For any nonnegative integer $n$, we have
	\begin{align}\label{23}		\sum_{k=0}^n\sum_{l=0}^{n}\binom{n}{k}\binom{n+k}{k}\binom{n}{l}\binom{n+l}{l}\dfrac{(-2)^{k+l}}{(k+l+1)\binom{k+l}{k}}=\dfrac{(-1)^n}{2n+1}.
	\end{align}
\end{lemma}

\begin{proof}
Note that
$$
\f{1}{(k+l+1)\binom{k+l}{k}}=\f{\Gamma(k+1)\Gamma(l+1)}{\Gamma(k+l+2)}=B(k+1,l+1)=\int_0^1 x^k(1-x)^l\d x,
$$
where $\Gamma$ is the Gamma function and $B$ is the Beta function. Therefore, the left-hand side of \eqref{23} becomes
\begin{align*}
&\sum_{k=0}^n\sum_{l=0}^{n}\binom{n}{k}\binom{n+k}{k}\binom{n}{l}\binom{n+l}{l}\int_0^1(-2x)^k(2x-2)^l\d x\\
&\quad=\int_0^1\l(\sum_{k=0}^n\binom{n}{k}\binom{n+k}{k}(-2x)^k\sum_{l=0}^{n}\binom{n}{l}\binom{n+l}{l}(2x-2)^l\r)\d x.
\end{align*}
By Pfaff's transformation (cf. \cite[p. 79, (2.3.14)]{AAR99}), we have
\begin{equation}\label{pfaff}
\sum_{k=0}^n\binom{n}{k}\binom{n+k}{k}(-z)^k=(-1)^n\sum_{k=0}^n\binom{n}{k}\binom{n+k}{k}(z-1)^k.
\end{equation}
It follows that
\begin{align*}
&\sum_{k=0}^n\binom{n}{k}\binom{n+k}{k}(-2x)^k\sum_{l=0}^{n}\binom{n}{l}\binom{n+l}{l}(2x-2)^l\\
&\quad=\sum_{k=0}^n\binom{n}{k}\binom{n+k}{k}(2x-1)^k\sum_{l=0}^n\binom{n}{l}\binom{n+l}{l}(1-2x)^l\\
&\quad=\sum_{k=0}^n\sum_{l=0}^n\binom{n}{k}\binom{n+k}{k}\binom{n}{l}\binom{n+l}{l}(-1)^l(2x-1)^{k+l}.
\end{align*}
Since
$$
\int_0^1 (2x-1)^{k+l}\d x=\f{1+(-1)^{k+l}}{2(k+l+1)},
$$
we can further transform the left-hand side of \eqref{23} into
\begin{equation}\label{Lem2.2key}
\sum_{k=0}^n\sum_{l=0}^n\binom{n}{k}\binom{n+k}{k}\binom{n}{l}\binom{n+l}{l}\f{(-1)^k+(-1)^l}{2(k+l+1)}.
\end{equation}
Letting $x=k+1$ in the following partial fraction decomposition:
\begin{equation}\label{pfd}
\sum_{l=0}^n\f{(-1)^l}{x+l}\binom{n+l}{l}\binom{n}{l}=\f{(-x+1)_n}{(x)_{n+1}},
\end{equation}
we find that
$$
\sum_{l=0}^n\f{(-1)^l}{k+l+1}\binom{n+l}{l}\binom{n}{l}=\f{(-k)_n}{(k+1)_{n+1}}=\begin{cases}0,\quad &\t{if}\ k<n,\\ \f{(-1)^n}{(2n+1)\binom{2n}{n}},\quad&\t{if}\ k=n, \end{cases}
$$
where $(x)_k=x(x+1)\cdots(x+k-1)$ stands for the Pochhammer symbol. Therefore,
\begin{equation}\label{Lem2.2key'}
\sum_{k=0}^n\sum_{l=0}^n\binom{n}{k}\binom{n+k}{k}\binom{n}{l}\binom{n+l}{l}\f{(-1)^l}{k+l+1}=\f{(-1)^n}{2n+1}.
\end{equation}
Similarly, we have
\begin{equation}\label{Lem2.2key''}
\sum_{k=0}^n\sum_{l=0}^n\binom{n}{k}\binom{n+k}{k}\binom{n}{l}\binom{n+l}{l}\f{(-1)^k}{k+l+1}=\f{(-1)^n}{2n+1}.
\end{equation}
Substituting \eqref{Lem2.2key'} and \eqref{Lem2.2key''} into \eqref{Lem2.2key}, we finally arrive at the desired result.
\end{proof}

\begin{lemma}\label{sigma1_1}
For any odd prime $p$ and $p$-adic integer $x$ with $m<(p-1)/2$, we have
\begin{align*}
&p\sum_{k=0}^m\sum_{l=0}^m\binom{x}{k}\binom{x+k}{k}\binom{x}{l}\binom{x+l}{l}2^{k+l}\sum_{n=0}^{k+l}\f{1}{n+1}\binom{n}{l}\binom{l}{n-k}\binom{p-1}{n}\\
&\quad\eq \f{p(-1)^m}{2m+1}\pmod{p^2}.
\end{align*}
\end{lemma}
\begin{proof}
Since $0\leq k,l\leq m<(p-1)/2$, we have $1\leq n+1\leq k+l+1< p-1$. By Lemma \ref{xkx},
\begin{align*}
&p\sum_{k=0}^m\sum_{l=0}^m\binom{x}{k}\binom{x+k}{k}\binom{x}{l}\binom{x+l}{l}2^{k+l}\sum_{n=0}^{k+l}\f{1}{n+1}\binom{n}{l}\binom{l}{n-k}\binom{p-1}{n}\\
&\quad\eq p\sum_{k=0}^m\sum_{l=0}^m\binom{m}{k}\binom{m+k}{k}\binom{m}{l}\binom{m+l}{l}2^{k+l}\sum_{n=0}^{k+l}\f{1}{n+1}\binom{n}{l}\binom{l}{n-k}\binom{p-1}{n}\\
&\quad\eq
p\sum_{k=0}^m\sum_{l=0}^m\binom{m}{k}\binom{m+k}{k}\binom{m}{l}\binom{m+l}{l}2^{k+l}\sum_{n=0}^{k+l}\f{(-1)^n}{n+1}\binom{n}{l}\binom{l}{n-k}\pmod{p^2}.
\end{align*}
From \cite[(2.2)]{Liu2018}, we know
$$
\sum_{n=0}^{k+l}\f{(-1)^n}{n+1}\binom{n}{l}\binom{l}{n-k}=\f{(-1)^{l+k}}{(l+k+1)\binom{l+k}{l}}.
$$
This, together with Lemma \ref{identity3}, gives
\begin{align*}
&p\sum_{k=0}^m\sum_{l=0}^m\binom{m}{k}\binom{m+k}{k}\binom{m}{l}\binom{m+l}{l}2^{k+l}\sum_{n=0}^{k+l}\f{(-1)^n}{n+1}\binom{n}{l}\binom{l}{n-k}\notag\\
&\quad\eq p\sum_{k=0}^m\sum_{l=0}^m\binom{m}{k}\binom{m+k}{k}\binom{m}{l}\binom{m+l}{l}\f{(-2)^{l+k}}{(l+k+1)\binom{l+k}{l}}\notag\\
&\quad=\f{p(-1)^m}{2m+1}\pmod{p^2}.
\end{align*}
This concludes the proof.
\end{proof}

\begin{lemma}\label{sigma1_2}
For any odd prime $p$ and $p$-adic integer $x$ with $m<(p-1)/2$, we have
\begin{align*}
&p\sum_{k=0}^m\sum_{l=m+1}^{p-m-1}\binom{x}{k}\binom{x+k}{k}\binom{x}{l}\binom{x+l}{l}2^{k+l}\sum_{n=0}^{k+l}\f{1}{n+1}\binom{n}{l}\binom{l}{n-k}\binom{p-1}{n}\\
&\quad\eq \f{pt(-1)^m}{2m+1}\pmod{p^2}.
\end{align*}
\end{lemma}

\begin{proof}
In view of Lemma \ref{xkx}, we deduce that
\begin{align*}
&p\sum_{k=0}^m\sum_{l=m+1}^{p-m-1}\binom{x}{k}\binom{x+k}{k}\binom{x}{l}\binom{x+l}{l}2^{k+l}\sum_{n=0}^{k+l}\f{1}{n+1}\binom{n}{l}\binom{l}{n-k}\binom{p-1}{n}\\
&\quad\eq p^2t\sum_{k=0}^m\sum_{l=m+1}^{p-m-1}\f{(-1)^{m+l+1}\binom{m}{k}\binom{m+k}{k}\binom{m+l}{l}}{(l-m)\binom{l}{m}}2^{k+l}\sum_{n=0}^{k+l}\f{1}{n+1}\binom{n}{l}\binom{l}{n-k}\binom{p-1}{n}\\
&\quad\eq\f{-(-1)^mpt\binom{2m}{m}\binom{p-1}{m}2^{p-1}}{(p-2m-1)\binom{p-1-m}{m}}\\
&\quad\eq
\f{(-1)^mpt}{2m+1}\pmod{p^2}.
\end{align*}
This proves the desired result.
\end{proof}

\medskip

\begin{proof}[Proof of Theorem \ref{mainth1}] We divide the proof into three cases.

{\it Case 1}. $m<(p-1)/2$.

From \cite[(3.5)]{Liu2018}, we know
\begin{align*}
\sum_{n=0}^{p-1}\binom{n}{k}\binom{n}{l}=p\sum_{n=0}^{k+l}\f{1}{n+1}\binom{n}{l}\binom{l}{n-k}\binom{p-1}{n}.
\end{align*}
Therefore,
\begin{align*}
\sum_{n=0}^{p-1}t_n(x)^2&=\sum_{n=0}^{p-1}\sum_{l=0}^n\sum_{k=0}^n\binom{n}{k}\binom{n}{l}\binom{x}{k}\binom{x+k}{k}\binom{x}{l}\binom{x+l}{l}2^{k+l}\\
&=\sum_{k=0}^{p-1}\sum_{l=0}^{p-1}\binom{x}{k}\binom{x+k}{k}\binom{x}{l}\binom{x+l}{l}2^{k+l}\sum_{n=0}^{p-1}\binom{n}{k}\binom{n}{l}\\
&=p\sum_{k=0}^{p-1}\sum_{l=0}^{p-1}\binom{x}{k}\binom{x+k}{k}\binom{x}{l}\binom{x+l}{l}2^{k+l}\sum_{n=0}^{k+l}\f{1}{n+1}\binom{n}{l}\binom{l}{n-k}\binom{p-1}{n}\\
&=\sum_{s=1}^9\sigma_s,
\end{align*}
where
\begin{gather*}
\sigma_1=\sum_{k=0}^m\sum_{l=0}^mf(k,l),\quad \sigma_2=\sum_{k=0}^m\sum_{l=m+1}^{p-1-m}f(k,l),\quad \sigma_3=\sum_{k=0}^m\sum_{l=p-m}^{p-1}f(k,l),\\
\sigma_4=\sum_{k=m+1}^{p-1-m}\sum_{l=0}^mf(k,l),\quad \sigma_5=\sum_{k=m+1}^{p-1-m}\sum_{l=m+1}^{p-1-m}f(k,l),\quad \sigma_6=\sum_{k=m+1}^{p-1-m}\sum_{l=p-m}^{p-1}f(k,l),\\
\sigma_7=\sum_{k=p-m}^{p-1}\sum_{l=0}^mf(k,l),\quad \sigma_8=\sum_{k=p-m}^{p-1}\sum_{l=m+1}^{p-1-m}f(k,l),\quad \sigma_9=\sum_{k=p-m}^{p-1}\sum_{l=p-m}^{p-1}f(k,l),
\end{gather*}
and $f(k,l)\ (0\leq k,l\leq p-1)$ stand for the summands
$$
p\binom{x}{k}\binom{x+k}{k}\binom{x}{l}\binom{x+l}{l}2^{k+l}\sum_{n=0}^{k+l}\f{1}{n+1}\binom{n}{l}\binom{l}{n-k}\binom{p-1}{n}.
$$
By the symmetry of $k$ and $l$, we have $\sigma_2=\sigma_4$. Moreover, in view of Lemma \ref{xkx}, we immediately obtain $\sigma_3\eq\sigma_5\eq\sigma_6\eq\sigma_7\eq\sigma_8\eq\sigma_9\eq0\pmod{p^2}$. Therefore,
$$
\sum_{n=0}^{p-1}t_n(x)^2\eq \sigma_1+2\sigma_2\pmod{p^2}.
$$
This, together with Lemmas \ref{sigma1_1} and \ref{sigma1_2}, gives
\begin{equation*}
\sum_{n=0}^{p-1}t_n(x)^2\eq \f{(-1)^m(p+2pt)}{2m+1}\eq \f{(-1)^{\<x\>_p}(p+2(x-\<x\>_p))}{2x+1}\pmod{p^2},
\end{equation*}
as desired.

\medskip

{\it Case 2}. $m>(p-1)/2$.

Clearly, $t_n(x)=t_n(-1-x)$ for all nonnegative integers $n$. So
$$
\sum_{n=0}^{p-1}t_n(x)^2=\sum_{n=0}^{p-1}t_n(-1-x)^2.
$$
Since $\<-1-x\>_{p}=p-1-m<(p-1)/2$, by Theorem \ref{mainth1} in Case 1, we have
\begin{align*}
\sum_{n=0}^{p-1}t_n(x)^2&\eq (-1)^{\<-1-x\>_{p}}\f{p+2(-1-x-\<-1-x\>_{p})}{2(-x-1)+1}\\
&=(-1)^{\<x\>_p}\f{p+2(x-\<x\>_p)}{2x+1}\pmod{p^2}.
\end{align*}
This proves Theorem \ref{mainth1} in Case 2.

\medskip

{\it Case 3}. $m=(p-1)/2$.

In this case, $m+1>p-1-m$. So we have
$$
\sum_{n=0}^{p-1}t_n(x)^2\eq \sigma_1\pmod{p^2},
$$
where $\sigma_1$ is defined as in Case 1. In view of Lemmas \ref{xkx} and \ref{identity3}, we have
\begin{align*} &\sigma_1=p\sum_{k=0}^m\sum_{l=0}^m\binom{x}{k}\binom{x+k}{k}\binom{x}{l}\binom{x+l}{l}2^{k+l}\sum_{n=0}^{k+l}\f{1}{n+1}\binom{n}{l}\binom{l}{n-k}\binom{p-1}{n}\notag\\
	&\quad\eq
	p\sum_{k=0}^m\sum_{l=0}^m\binom{m}{k}\binom{m+k}{k}\binom{m}{l}\binom{m+l}{l}\dfrac{(-2)^{k+l}}{(k+l+1)\binom{k+l}{k}}\notag\\
	&\quad
	=\f{p(-1)^m}{2m+1}\\
&\quad =\l(\f{-1}{p}\r)\pmod{p^2}.
\end{align*}
This proves Theorem \ref{mainth1} in Case 3.

The proof of Theorem \ref{mainth1} is now complete.
\end{proof}

\section{Proofs of Theorem \ref{mainth2} and Corollary \ref{sunconm}}
\setcounter{lemma}{0}

To show Theorem \ref{mainth2}, we need the following preliminary results.
\begin{lemma}\label{identity2}
	For nonnegative integers $k$ and $l$, we have
	$$ \sum_{n=0}^{k+l}\f{(-1)^n}{n+2}\binom{n}{l}\binom{l}{n-k}=\f{(-1)^{l+k}}{\binom{k+l+2}{l+1}}.
	$$
	\begin{proof}
		Letting $x=2$ in the following partial fraction decomposition
		\begin{align*}
			\sum_{n=0}^{k+l}\dfrac{(-1)^{l+k+n}}{x+n}\binom{n}{l}\binom{l}{n-k}=\frac{{(x)_l}{(x)_k}}{{(x)_{l+k+1}}},
		\end{align*}
	we are done.
	\end{proof}
\end{lemma}

We also need the following identity similar to Lemma \ref{identity3}.

\begin{lemma}\label{identity4}
	For any nonnegative integer $n$, we have
	\begin{align}\label{id32}
		\sum_{k=0}^n\sum_{l=0}^{n}\binom{n}{k}\binom{n+k}{k}\binom{n}{l}\binom{n+l}{l}\dfrac{(-2)^{k+l}}{\binom{k+l+2}{k+1}}=\f14-\f{(-1)^n(2n^2+2n-1)}{8n+4}.
	\end{align}
\end{lemma}
\begin{proof}
The proof proceeds similarly as the one of Lemma \ref{identity3}. It is easy to see that
\begin{align*}
\f{1}{\binom{k+l+2}{k+1}}&=\f{(k+l+3)\Gamma(k+2)\Gamma(l+2)}{\Gamma(k+l+4)}=(k+l+3)B(k+2,l+2)\\
&=(k+l+3)\int_0^1 x^{k+1}(1-x)^{l+1}\d x.
\end{align*}
Then the left-hand side of \eqref{id32} becomes
\begin{align}\label{id32lhs}
&\sum_{k=0}^n\sum_{l=0}^{n}\binom{n}{k}\binom{n+k}{k}\binom{n}{l}\binom{n+l}{l}(k+l+3)(-2)^{k+l}\int_0^1 x^{k+1}(1-x)^{l+1}\d x\notag\\
&\quad=\f14\sum_{k=0}^n\sum_{l=0}^{n}\binom{n}{k}\binom{n+k}{k}\binom{n}{l}\binom{n+l}{l}(k+l+3)\int_0^1(-2x)^{k+1}(2x-2)^{l+1}\d x.
\end{align}

Differentiating both sides of \eqref{pfaff} with respect to $z$, we obtain
\begin{equation}\label{pfaffdiff}
\sum_{k=1}^n\binom{n}{k}\binom{n+k}{k}k(-z)^{k-1}=(-1)^{n-1}\sum_{k=1}^n\binom{n}{k}\binom{n+k}{k}k(z-1)^{k-1}.
\end{equation}
With the help of \eqref{pfaff} and \eqref{pfaffdiff}, we get
\begin{align*}
&\sum_{k=0}^n\sum_{l=0}^{n}\binom{n}{k}\binom{n+k}{k}\binom{n}{l}\binom{n+l}{l}k\int_0^1(-2x)^{k+1}(2x-2)^{l+1}\d x\notag\\
&\quad=\int_0^14x^2(2x-2)\l(\sum_{k=1}^n\binom{n}{k}\binom{n+k}{k}k(-2x)^{k-1}\sum_{l=0}^{n}\binom{n}{l}\binom{n+l}{l}(2x-2)^{l}\r)\d x\\
&\quad=-\int_0^14x^2(2x-2)\l(\sum_{k=1}^n\binom{n}{k}\binom{n+k}{k}k(2x-1)^{k-1}\sum_{l=0}^{n}\binom{n}{l}\binom{n+l}{l}(1-2x)^{l}\r)\d x\\
&\quad=-\sum_{k=1}^n\sum_{l=0}^{n}\binom{n}{k}\binom{n+k}{k}\binom{n}{l}\binom{n+l}{l}k\int_0^14x^2(2x-2)(2x-1)^{k-1}(1-2x)^l\d x.
\end{align*}
It is routine to evaluate that
\begin{align*}
&\int_0^14x^2(2x-2)(2x-1)^{k-1}(1-2x)^l\d x\\
&\quad=\f{(-1)^k-(-1)^{l}}{2(k+l)}-\f{(-1)^k+(-1)^{l}}{2(k+l+1)}-\f{(-1)^k-(-1)^{l}}{2(k+l+2)}+\f{(-1)^k+(-1)^{l}}{2(k+l+3)}.
\end{align*}
In view of \eqref{pfd}, we have
\begin{align*}
&-\sum_{k=1}^n\sum_{l=0}^{n}\binom{n}{k}\binom{n+k}{k}\binom{n}{l}\binom{n+l}{l}\f{k}{2(k+l)}\l((-1)^k-(-1)^l\r)\\
&\quad=\f12\sum_{k=1}^n\sum_{l=0}^{n}\binom{n}{k}\binom{n+k}{k}\binom{n}{l}\binom{n+l}{l}\f{k(-1)^l}{k+l}\\
&\qquad-\f12\sum_{k=1}^n\sum_{l=0}^{n}\binom{n}{k}\binom{n+k}{k}\binom{n}{l}\binom{n+l}{l}(-1)^k\\
&\qquad+\f12\sum_{k=1}^n\sum_{l=0}^{n}\binom{n}{k}\binom{n+k}{k}\binom{n}{l}\binom{n+l}{l}\f{l(-1)^k}{k+l}\\
&\quad=0-\f{(-1)^n-1}{2}\sum_{l=0}^n\binom{n}{l}\binom{n+l}{l}-\f{1}{2}\sum_{l=1}^n\binom{n}{l}\binom{n+l}{l}\\
&\quad=\f12-\f{(-1)^n}{2}\sum_{l=0}^n\binom{n}{l}\binom{n+l}{l}.
\end{align*}
Similarly,
\begin{align*}
&-\sum_{k=1}^n\sum_{l=0}^{n}\binom{n}{k}\binom{n+k}{k}\binom{n}{l}\binom{n+l}{l}\f{k}{2(k+l+1)}\l((-1)^k+(-1)^l\r)\\
&\quad=\f{(-1)^n}{4n+2}-\f{(-1)^n}{2}\sum_{l=0}^n\binom{n}{l}\binom{n+l}{l},\\
&-\sum_{k=1}^n\sum_{l=0}^{n}\binom{n}{k}\binom{n+k}{k}\binom{n}{l}\binom{n+l}{l}\f{k}{2(k+l+2)}\l((-1)^k-(-1)^l\r)\\
&\quad=\f{(-1)^n(2n^2+2n+1)}{4n+2}-\f{(-1)^n}{2}\sum_{l=0}^n\binom{n}{l}\binom{n+l}{l},\\
&-\sum_{k=1}^n\sum_{l=0}^{n}\binom{n}{k}\binom{n+k}{k}\binom{n}{l}\binom{n+l}{l}\f{k}{2(k+l+3)}\l((-1)^k+(-1)^l\r)\\
&\quad=\f{3(-1)^n(2n^2(n+1)^2-1)}{(4n+2)(2n+3)(2n-1)}-\f{(-1)^n}{2}\sum_{l=0}^n\binom{n}{l}\binom{n+l}{l}.
\end{align*}
Combining the above, we arrive at
\begin{align}\label{id32key'}
&\sum_{k=0}^n\sum_{l=0}^{n}\binom{n}{k}\binom{n+k}{k}\binom{n}{l}\binom{n+l}{l}k\int_0^1(-2x)^{k+1}(2x-2)^{l+1}\d x\notag\\
&\quad=\f12-\f{(-1)^n}{4n+2}-\f{(-1)^n(2n^2+2n+1)}{4n+2}+\f{3(-1)^n(2n^2(n+1)^2-1)}{(4n+2)(2n+3)(2n-1)}.
\end{align}
By the symmetry of $k$ and $l$, we also have
\begin{align}\label{id32key''}
&\sum_{k=0}^n\sum_{l=0}^{n}\binom{n}{k}\binom{n+k}{k}\binom{n}{l}\binom{n+l}{l}l\int_0^1(-2x)^{k+1}(2x-2)^{l+1}\d x\notag\\
&\quad=\f12-\f{(-1)^n}{4n+2}-\f{(-1)^n(2n^2+2n+1)}{4n+2}+\f{3(-1)^n(2n^2(n+1)^2-1)}{(4n+2)(2n+3)(2n-1)}.
\end{align}

Using \eqref{pfaff} again, we obtain
\begin{align*}
&\sum_{k=0}^n\sum_{l=0}^{n}\binom{n}{k}\binom{n+k}{k}\binom{n}{l}\binom{n+l}{l}\int_0^1(-2x)^{k+1}(2x-2)^{l+1}\d x\\
&\quad=-\int_0^12x(2x-2)\l(\sum_{k=0}^n\binom{n}{k}\binom{n+k}{k}(-2x)^{k}\sum_{l=0}^{n}\binom{n}{l}\binom{n+l}{l}(2x-2)^{l}\r)\d x\\
&\quad=-\int_0^142x(2x-2)\l(\sum_{k=0}^n\binom{n}{k}\binom{n+k}{k}(2x-1)^{k}\sum_{l=0}^{n}\binom{n}{l}\binom{n+l}{l}(1-2x)^{l}\r)\d x\\
&\quad=-\sum_{k=0}^n\sum_{l=0}^{n}\binom{n}{k}\binom{n+k}{k}\binom{n}{l}\binom{n+l}{l}\int_0^12x(2x-2)(2x-1)^{k}(1-2x)^l\d x\\
&\quad=\f{1}{2}\sum_{k=0}^n\sum_{l=0}^{n}\binom{n}{k}\binom{n+k}{k}\binom{n}{l}\binom{n+l}{l}\l(\f{(-1)^k+(-1)^l}{k+l+1}-\f{(-1)^k+(-1)^l}{k+l+3}\r).
\end{align*}
Then, by \eqref{pfd} and noticing the symmetry of $k$ and $l$, we deduce that
\begin{align}\label{id32key'''}
&\sum_{k=0}^n\sum_{l=0}^{n}\binom{n}{k}\binom{n+k}{k}\binom{n}{l}\binom{n+l}{l}\int_0^1(-2x)^{k+1}(2x-2)^{l+1}\d x\notag\\
&\quad=\f{(-1)^n}{2n+1}-\f{(-1)^n(2n^2(n+1)^2-1)}{(2n+1)(2n+3)(2n-1)}.
\end{align}

Finally, substituting \eqref{id32key'}--\eqref{id32key'''} into \eqref{id32lhs} gives
\begin{align*}
&\sum_{k=0}^n\sum_{l=0}^{n}\binom{n}{k}\binom{n+k}{k}\binom{n}{l}\binom{n+l}{l}\dfrac{(-2)^{k+l}}{\binom{k+l+2}{k+1}}=\f14-\f{(-1)^n(2n^2+2n-1)}{8n+4},
\end{align*}
as desired.
\end{proof}

\begin{lemma}\label{sigma2_1}
	For any odd prime $p$ and $p$-adic integer $x$ with $m<(p-1)/2$, we have
	\begin{align*}
		&\sum_{k=0}^m\sum_{l=0}^m\binom{x}{k}\binom{x+k}{k}\binom{x}{l}\binom{x+l}{l}2^{k+l}\sum_{n=0}^{k+l}\f{p(p+1)}{n+2}\binom{n}{l}\binom{l}{n-k}\binom{p-1}{n}\\
		&\quad\eq
		\f p4-\f{(-1)^m(2m^2+2m-1)p}{8m+4}\pmod{p^2}.
	\end{align*}
\end{lemma}

\begin{proof}
	Since $0\leq k,l\leq m<(p-1)/2$, we have $1\leq n+2\leq k+l+2\leq p-1$. By Lemma \ref{xkx},
	\begin{align*}
		&\sum_{k=0}^m\sum_{l=0}^m\binom{x}{k}\binom{x+k}{k}\binom{x}{l}\binom{x+l}{l}2^{k+l}\sum_{n=0}^{k+l}\f{p(p+1)}{n+2}\binom{n}{l}\binom{l}{n-k}\binom{p-1}{n}\notag\\
		&\quad\eq p\sum_{k=0}^m\sum_{l=0}^m\binom{m}{k}\binom{m+k}{k}\binom{m}{l}\binom{m+l}{l}2^{k+l}\sum_{n=0}^{k+l}\f{(-1)^n}{n+2}\binom{n}{l}\binom{l}{n-k}\pmod{p^2}.
	\end{align*}
	Moreover, with the help of Lemmas \ref{identity2} and \ref{identity4}, we get
	\begin{align*}\label{sigma2_1decom1}
		&p\sum_{k=0}^m\sum_{l=0}^m\binom{m}{k}\binom{m+k}{k}\binom{m}{l}\binom{m+l}{l}2^{k+l}\sum_{n=0}^{k+l}\f{(-1)^n}{n+2}\binom{n}{l}\binom{l}{n-k}\notag\\
		&\quad= p\sum_{k=0}^m\sum_{l=0}^m\binom{m}{k}\binom{m+k}{k}\binom{m}{l}\binom{m+l}{l}\f{(-2)^{l+k}}{\binom{k+l+2}{l+1}}\notag\\
		&\quad=
		\f p4-\f{(-1)^m(2m^2+2m-1)p}{8m+4}.
	\end{align*}
Combining the above, we are done.
\end{proof}

\begin{lemma}\label{sigma2_2}
	For any odd prime $p$ and $p$-adic integer $x$ with $m<(p-1)/2$, we have
	\begin{align*}
		&\sum_{k=0}^m\sum_{l=m+1}^{p-m-1}\binom{x}{k}\binom{x+k}{k}\binom{x}{l}\binom{x+l}{l}2^{k+l}\sum_{n=0}^{k+l}\f{p(p+1)}{n+2}\binom{n}{l}\binom{l}{n-k}\binom{p-1}{n}\\
		&\quad\eq \f{pt(-1)^m(1-2m^2-2m)}{8m+4}\pmod{p^2}.
	\end{align*}
\end{lemma}

\begin{proof}	
	By Lemma \ref{xkx}, for $m+1\leq k\leq p-m-1$, we have
	\begin{align*}
		&\binom{x}{k}\binom{x+k}{k}\eq \f{pt(-1)^{m+k+1}\binom{m+k}{k}}{(k-m)\binom{k}{m}}\pmod{p^2}.
	\end{align*}
	Hence we get
	\begin{align*}
		&\sum_{k=0}^m\sum_{l=m+1}^{p-m-1}\binom{x}{k}\binom{x+k}{k}\binom{x}{l}\binom{x+l}{l}2^{k+l}\sum_{n=0}^{k+l}\f{p(p+1)}{n+2}\binom{n}{l}\binom{l}{n-k}\binom{p-1}{n}\\
		&\quad\eq -\binom{x}{m}\binom{x+m}{m}\binom{x}{p-2-m}\binom{x+p-2-m}{p-2-m}2^{p-2}\binom{p-2}{m}\\
		&\qquad-\binom{x}{m-1}\binom{x+m-1}{m-1}\binom{x}{p-1-m}\binom{x+p-1-m}{p-1-m}2^{p-2}\binom{p-2}{m-1}\\
		&\qquad+\binom{x}{m}\binom{x+m}{m}\binom{x}{p-1-m}\binom{x+p-1-m}{p-1-m}2^{p-1}\binom{p-2}{m-1}(1+m-p)\\
		&\quad\eq
		-\binom{2m}{m}\f{pt{\binom{p-2}{m}}^22^{p-2}}{(p-2m-2)\binom{p-m-2}{m}}+m\binom{2m-1}{m-1}\f{pt\binom{p-1}{m}\binom{p-2}{m-1}2^{p-2}}{(p-2m-1)\binom{p-m-1}{m}}\\
		&\qquad-(1+m)\binom{2m}{m}\f{pt\binom{p-1}{m}\binom{p-2}{m-1}2^{p-1}}{(p-2m-1)\binom{p-m-1}{m}}\\
		&\quad\eq
		\f{pt(-1)^m(m+1)^2}{4(2m+1)}+\f{pt(-1)^mm^2}{4(2m+1)}+\f{pt(-1)^{m-1}(m+1)m}{2m+1}\\
		&\quad
		=\f{pt(-1)^m(1-2m^2-2m)}{8m+4}\pmod{p^2},
	\end{align*}
as desired.
\end{proof}

\begin{proof}[Proof of Theorem \ref{mainth2}] We divide the proof into three cases.
	
	{\it Case 1}. $m<(p-1)/2$.
	
 For nonnegative integers $k$ and $l$, we have 
\begin{align*}
	\sum_{n=0}^{p-1}(n+1)\binom{n}{k}\binom{n}{l}&=\sum_{i=0}^{l}\binom{i+k}{l}\binom{l}{i}\sum_{n=0}^{p-1}(n+1)\binom{n}{i+k}\\
	&=\sum_{i=0}^{l}\binom{i+k}{l}\binom{l}{i}(i+k+1)\sum_{n=0}^{p}\binom{n}{i+k+1}\\
	&=\sum_{i=0}^{l}\binom{i+k}{l}\binom{l}{i}(i+k+1)\binom{p+1}{i+k+2}\\
    &=\sum_{n=0}^{k+l}\dfrac{p(p+1)}{n+2}\binom{n}{l}\binom{l}{n-k}\binom{p-1}{n}.
\end{align*}
	Therefore,
	\begin{align*}
		\sum_{n=0}^{p-1}(n+1)t_n(x)^2&=\sum_{n=0}^{p-1}\sum_{l=0}^n\sum_{k=0}^n\binom{n}{k}\binom{n}{l}\binom{x}{k}\binom{x+k}{k}\binom{x}{l}\binom{x+l}{l}2^{k+l}\\
		&=\sum_{k=0}^{p-1}\sum_{l=0}^{p-1}\binom{x}{k}\binom{x+k}{k}\binom{x}{l}\binom{x+l}{l}2^{k+l}\sum_{n=0}^{p-1}(n+1)\binom{n}{k}\binom{n}{l}\\
		&=\sum_{k=0}^{p-1}\sum_{l=0}^{p-1}\binom{x}{k}\binom{x+k}{k}\binom{x}{l}\binom{x+l}{l}2^{k+l}\sum_{n=0}^{k+l}\dfrac{p(p+1)}{n+2}\binom{n}{l}\binom{l}{n-k}\binom{p-1}{n}\\
		&=\sum_{s=1}^9\tau_s,
	\end{align*}
	where
	\begin{gather*}
		\tau_1=\sum_{k=0}^m\sum_{l=0}^mg(k,l),\quad \tau_2=\sum_{k=0}^m\sum_{l=m+1}^{p-1-m}g(k,l),\quad \tau_3=\sum_{k=0}^m\sum_{l=p-m}^{p-1}g(k,l),\\
		\tau_4=\sum_{k=m+1}^{p-1-m}\sum_{l=0}^mg(k,l),\quad \tau_5=\sum_{k=m+1}^{p-1-m}\sum_{l=m+1}^{p-1-m}g(k,l),\quad \tau_6=\sum_{k=m+1}^{p-1-m}\sum_{l=p-m}^{p-1}g(k,l),\\
		\tau_7=\sum_{k=p-m}^{p-1}\sum_{l=0}^mg(k,l),\quad \tau_8=\sum_{k=p-m}^{p-1}\sum_{l=m+1}^{p-1-m}g(k,l),\quad \tau_9=\sum_{k=p-m}^{p-1}\sum_{l=p-m}^{p-1}g(k,l),
	\end{gather*}
	and $g(k,l)\ (0\leq k,l\leq p-1)$ stand for the summands
	$$
	\binom{x}{k}\binom{x+k}{k}\binom{x}{l}\binom{x+l}{l}2^{k+l}\sum_{n=0}^{k+l}\f{p(p+1)}{n+2}\binom{n}{l}\binom{l}{n-k}\binom{p-1}{n}.
	$$
	By the symmetry of $k$ and $l$, we have $\tau_2=\tau_4$. Meanwhile, in view of Lemma \ref{xkx}, we immediately obtain $\tau_3\eq\tau_5\eq\tau_6\eq\tau_7\eq\tau_8\eq\tau_9\eq0\pmod{p^2}$. Therefore
	\begin{equation*}
		\sum_{n=0}^{p-1}(n+1)t_n(x)^2\eq \tau_1+2\tau_2\pmod{p^2}.
	\end{equation*}
	Then, by Lemmas \ref{sigma2_1} and \ref{sigma2_2} we obtain
	\begin{align*}
	\sum_{k=0}^{p-1}(n+1)t_n(x)^2&\equiv \f{p}{4}-\f{(-1)^m(2m^2+2m-1)(2t+1)p}{8m+4}\\
&\equiv \f{p}{4}-\f{(-1)^{\<x\>_p}(2x^2+2x-1)(p+2(x-\<x\>_p))}{8x+4}\pmod{p^2}.
	\end{align*}
    This concludes the proof of Theorem \ref{mainth2} in Case 1.
	
	\medskip
	
	{\it Case 2}. $m>(p-1)/2$.
	
	Clearly, $t_n(x)=t_n(-1-x)$ for each nonnegative integer $n$. So
	$$
	\sum_{n=0}^{p-1}(n+1)t_n(x)^2=\sum_{n=0}^{p-1}(n+1)t_n(-1-x)^2.
	$$
	Since $\<-1-x\>_{p}=p-1-m<(p-1)/2$, by Theorem \ref{mainth2} in Case 1 we have
	\begin{align*}
		&\sum_{n=0}^{p-1}(n+1)t_n(x)^2\\
		&\quad\eq \f{p}{4}-\f{(-1)^{\<-1-x\>_p}(2(-1-x)^2+2(-1-x)-1)(p+2(-1-x-\<-1-x\>_p))}{8(-1-x)+4}\\
&\quad=\f{p}{4}-\f{(-1)^{\<x\>_p}(2x^2+2x-1)(p+2(x-\<x\>_p))}{8x+4}\pmod{p^2}.
	\end{align*}
	This proves Theorem \ref{mainth2} in Case 2.
	
	\medskip
	
	{\it Case 3}. $m=(p-1)/2$.
	
	In this case, $m+1>p-1-m$. So we have
	\begin{equation}\label{c3_1}
		\sum_{n=0}^{p-1}(n+1)t_n(x)^2\eq \tau_1\pmod{p^2},
	\end{equation}
	where $\tau_1$ is defined as in Case 1.
	
	In view of Lemmas \ref{xkx} and \ref{identity4}, we have
	\begin{align}\label{c3_2}
		&\tau_1=\sum_{k=0}^m\sum_{l=0}^m\binom{x}{k}\binom{x+k}{k}\binom{x}{l}\binom{x+l}{l}2^{k+l}\sum_{n=0}^{k+l}\f{p(p+1)}{n+2}\binom{n}{l}\binom{l}{n-k}\binom{p-1}{n}\notag\\
		&\quad\eq
		p\sum_{k=0}^{m}\sum_{l=0}^{m}\binom{m}{k}\binom{m+k}{k}\binom{m}{l}\binom{m+l}{l}\dfrac{(-2)^{k+l}}{\binom{k+l+2}{k+1}}\notag\\
&\quad=\f{p}{4}-\f{(-1)^{m}(2m^2+2m-1)p}{8m+4}\notag\\
&\quad\eq \f{p}{4}+\f{3}{8}\l(\f{-1}{p}\r)\pmod{p^2}.
	\end{align}
    Combining \eqref{c3_1} and \eqref{c3_2}, we prove Theorem \ref{mainth2} in this case.

	The proof of Theorem \ref{mainth2} is now complete.
\end{proof}

\medskip

\noindent{\it Proof of Corollary \ref{sunconm}}. By Theorems \ref{mainth1} and \ref{mainth2} in the case $x=-1/2$, we immediately obtain \eqref{suncon1}.

Note that
\begin{align*}
\l\<-\f14\r\>_p&=\begin{cases}(p-1)/4,\quad &\t{if}\ p\eq1\pmod4,\\ (3p-1)/4,\quad &\t{if}\ p\eq3\pmod4,\end{cases}\\
\l\<-\f13\r\>_p&=\begin{cases}(p-1)/3,\quad &\t{if}\ p\eq1\pmod3,\\ (2p-1)/3,\quad &\t{if}\ p\eq2\pmod3,\end{cases}\\
\l\<-\f16\r\>_p&=\begin{cases}(p-1)/6,\quad &\t{if}\ p\eq1\pmod6,\\ (5p-1)/6,\quad &\t{if}\ p\eq5\pmod6.\end{cases}
\end{align*}
Then it is routine to verify \eqref{suncon2}--\eqref{suncon4} via combining the above and Theorems \ref{mainth1} and \ref{mainth2}.\qed

\begin{Ack}
This work is supported by the National Natural Science Foundation of China (grant 12201301).
\end{Ack}

\end{document}